\documentclass[12pt,a4paper,reqno]{amsart}
\usepackage{amsmath}
\usepackage{amsfonts}
\usepackage{amssymb}
\usepackage{amsthm}
\usepackage{enumerate}
\usepackage{centernot}
\usepackage{fullpage}
\usepackage{mathrsfs}
\usepackage{graphicx}
\usepackage[colorlinks,
linkcolor=blue,
anchorcolor=green,
citecolor=blue
]{hyperref}
\usepackage{color}

\newcommand{\arctanh}{\operatorname{arctanh}}

\theoremstyle{plain}

  \newtheorem{proposition}[]{Proposition}
  \newtheorem{lemma}[]{Lemma}
  \newtheorem{theorem}[]{Theorem}
  \newtheorem{corollary}[]{Corollary}
  
  \newtheorem{remark}[]{Remark}

\title[The free state for the Potts model]{The free state for the Potts model on Cayley trees is either extremal or glassy}

\author{Jianping Jiang}
\address{Yau Mathematical Sciences Center, Tsinghua University, Beijing 100084, China.}
\email{jianpingjiang@tsinghua.edu.cn}
\author{Sike Lang}
\address{Qiuzhen College, Tsinghua University, Beijing 100084, China.}
\email{langsk24@mails.tsinghua.edu.cn}

\begin{document}
\dedicatory{Dedicated to Chuck Newman on the occasion of his 80th birthday.}
\begin{abstract}
For the Potts model on the Cayley tree $\mathbb{T}^d$ with branching factor $d\geq 2$, we consider the free state which is obtained as the limiting Gibbs measure with free boundary conditions. We prove that the free state is either extremal or glassy (i.e., whose decomposition into extremal Gibbs measures contains uncountably many components). As a corollary, the free state for the Ising model on $\mathbb{T}^d$ is glassy if and only if the inverse temperature $\beta>\arctanh(1/\sqrt{d})$; this generalizes a previous result by Gandolfo, Maes, Ruiz and Shlosman (2020) from a very low temperature regime to the entire spin-glass regime.
\end{abstract}

\maketitle

\section{Introduction and main result}
Let $\mathbb{T}^d$ be the Cayley tree with branching factor $d\geq 2$, i.e., each vertex has degree $d+1$. Due to its rich mathematical formulation, the $q$-state Potts model on $\mathbb{T}^d$ has been studied extensively for decades \cite{Pre74,Geo11,Roz13,Roz21,Roz22}. 

Let us focus on the most studied case $q=2$ (the Ising model) for a moment. It is well-known that there is a unique Gibbs measure if and only if the inverse temperature $\beta \leq \arctanh(1/d)$ (see, e.g., \cite{Pre74}). The infinite-volume Gibbs measure with free boundary conditions is usually called the \textit{free state} (or the \textit{disordered state}). Bleher, Ruiz and Zagrebnov \cite{Ble90,BRZ95} proved that the free state is extremal if and only if $\beta \leq \arctanh(1/\sqrt{d})$; see also \cite{Iof96} for a different and elegant proof. It turns out that the free state is non-extremal if and only if the associated broadcasting process is reconstructible \cite{EKPS00,Mos04}. The next natural question is: what is the decomposition of the free state into extremal Gibbs measures? Gandolfo, Maes, Ruiz and Shlosman \cite{GMRS20} proved that a continuum of extremal states are contained in this decomposition when $\beta$ is sufficiently large; this result was later generalized to general $q$ and broader finite-spin models in \cite{CKLN26}. However, both \cite{GMRS20} and \cite{CKLN26} use the Peierls contour argument which requires $\beta$ to be large. In this paper, we overcome this barrier by proving a zero-one law for an overlap order parameter. As a consequence, we extend the continuum decomposition result (for the Ising model on $\mathbb{T}^d$) to the entire non-extremality regime $\beta>\arctanh(1/\sqrt{d})$.

For a finite subset $\Lambda$ in the vertex set $V(\mathbb{T}^d)$, the $q$-state Potts model (with $q \geq 2$) at inverse temperature $\beta$ on the subgraph induced by $\Lambda$ with boundary conditions $\eta \in \{1,\dots,q\}^{\Lambda^c}$ is the probability measure
\begin{equation}\label{eq:Pottsdef}
	\mu^{\eta}_{\Lambda,\beta}(\sigma)=\frac{1}{Z^{\eta}_{\Lambda,\beta}}\exp\left[\beta\sum_{u,v\in\Lambda: u\sim v}\delta_{\sigma_u,\sigma_v}+\beta\sum_{u\in\Lambda, v\in \Lambda^c: u\sim v}\delta_{\sigma_u,\eta_v}\right],\quad \forall \sigma \in \{1,\dots, q\}^{\Lambda},
\end{equation}
where $\delta_{i,j}$ is the Kronecker delta, $u\sim v$ denotes a nearest-neighbor edge $uv$ in the edge set $ E(\mathbb{T}^d)$, and $Z^{\eta}_{\Lambda,\beta}$ is the partition function. Let $\mu_{\Lambda,\beta}^{\varnothing}$ be the Potts model on $\Lambda$ with free boundary conditions (i.e., removing the second sum in \eqref{eq:Pottsdef}). By Kolmogorov's extension theorem (see, e.g., Theorem 6.6 of \cite{FV18}), there exists a unique measure $\mu_{\beta}^{\varnothing}$ on $\{1,\dots,q\}^{V(\mathbb{T}^d)}$ such that for each $\Lambda\Subset V(\mathbb{T}^d)$ (we use $\Subset$ for a finite subset) with its induced subgraph being connected,
\begin{equation}\label{eq:Kol}
	\mu_{\beta}^{\varnothing}(\sigma_u=s_u, \forall u \in \Lambda)=\mu_{\Lambda,\beta}^{\varnothing}(\sigma_u=s_u, \forall u \in \Lambda),\quad \forall  s\in \{1,\dots,q\}^{\Lambda}.
\end{equation}
 A measure $\mu$ on $\{1,\dots,q\}^{V(\mathbb{T}^d)}$ is called a \textit{Gibbs measure} if it satisfies the Dobrushin-Lanford-Ruelle (DLR) equations
\begin{equation}\label{eq:DLR}
	\mu(A)=\int \mu_{\Lambda,\beta}^{\eta}(A) \mu (d\eta), \quad \forall \Lambda\Subset V(\mathbb{T}^d), \forall \text{ local event } A \text{ in } \Lambda.
\end{equation}
Let $\mathscr{G}_{\beta}$ be the set of all Gibbs measures. One may check directly that the free state $\mu_{\beta}^{\varnothing}\in \mathscr{G}_{\beta}$. A measure $\mu \in \mathscr{G}_{\beta}$ is \textit{extremal} (or a \textit{pure state}) if it cannot be expressed as $\mu=\lambda \mu_1+(1-\lambda)\mu_2$ with $\lambda \in (0,1)$ and $\mu_1\neq \mu_2 \in \mathscr{G}_{\beta}$. The set of all extremal measures in $\mathscr{G}_{\beta}$  is denoted by $\text{ex}\mathscr{G}_{\beta}$. It is well-known (see, e.g., Theorem~6.72 of \cite{FV18} or Theorem~7.26 of \cite{Geo11}) that
\begin{equation}\label{eq:extremaldecom}
	\mu(\cdot) = \int_{\text{ex}\mathscr{G}_{\beta}} \nu(\cdot) \lambda_{\mu} (d\nu),\quad \forall \mu \in \mathscr{G}_{\beta},
\end{equation}
where $\lambda_{\mu}$ is the unique measure on $\text{ex}\mathscr{G}_{\beta}$ for which such a representation holds. Our main result is the following dichotomy for the extremal decomposition of $\mu_{\beta}^{\varnothing}$.

\begin{theorem}\label{thm}
	The free state for the Potts model on $\mathbb{T}^d$ with $d \geq 2$ for each $\beta \geq 0$ satisfies the following dichotomy:
	\begin{enumerate}[(i)]
		\item either $\mu_{\beta}^{\varnothing}$ is extremal, i.e., $\mu_{\beta}^{\varnothing}\in \textnormal{ex}\mathscr{G}_{\beta}$,
		\item or the extremal decomposition measure $\lambda_{\mu_{\beta}^{\varnothing}}$ has no atoms, i.e., $\lambda_{\mu_{\beta}^{\varnothing}}(\{\nu\})=0$ for each $\nu\in \textnormal{ex}\mathscr{G}_{\beta}$. In particular, the support of $\lambda_{\mu_{\beta}^{\varnothing}}$ contains uncountably many extremal measures.
	\end{enumerate}
\end{theorem}

\begin{remark}
	Our proof strategy for Theorem \ref{thm} can also be applied to the Potts model with uniform external field on $\mathbb{T}^d$. In Section \ref{sec:gen}, we show that the same dichotomy holds for all completely homogeneous splitting Gibbs measures defined in \cite{BR19} (also called translation-invariant splitting Gibbs measures in \cite{KRK14,KR17}).
\end{remark}

Following \cite{GMRS20}, we call a free state \textit{glassy} when it satisfies part (ii) of the theorem. Theorem \ref{thm} may be combined with the known results on extremality problem (or reconstruction problem) to find the exact threshold. For example, for the classical Ising model on $\mathbb{T}^d$ at the inverse temperature $\beta$, which is the $2$-state Potts model defined by \eqref{eq:Pottsdef} at $2\beta$,  combining Theorem \ref{thm} with \cite{Ble90,BRZ95,Iof96}, we obtain
\begin{corollary}\label{cor}
	The free state $\mu_{\beta}^{\varnothing}$  for the classical Ising model on $\mathbb{T}^d$ satisfies:
	\begin{enumerate}[(i)]
		\item if $0\leq \beta \leq \arctanh(1/\sqrt{d})$, then $\mu_{\beta}^{\varnothing}$  is extremal,
		\item if $\beta >\arctanh(1/\sqrt{d})$, then its extremal decomposition measure $\lambda_{\mu_{\beta}^{\varnothing}}$ has no atoms and thus is supported on uncountably many extremal measures.
	\end{enumerate}
\end{corollary}

The short-range spin glass model on $\mathbb{T}^d$ is defined by assigning i.i.d. symmetric $\pm 1$ couplings on all nearest-neighbor edges. For the spin glass on $\mathbb{T}^d$ with fixed boundary conditions, which is equivalent to the ferromagnetic Ising model on $\mathbb{T}^d$ with random boundary conditions by a gauge transformation, it was proved in \cite{CCST86} that the spin glass phase transition occurs at $\beta_{SG}=\arctanh(1/\sqrt{d})$. Since the existence of uncountably many extremal Gibbs states is a hallmark of the spin glass phase \cite{NS07}, our Corollary \ref{cor} suggests that one may already observe such a phase transition from the free state of the ferromagnetic Ising model; see Proposition 4.2 of \cite{GMRS20} for a rigorous justification of the spin glass transition. This should be contrasted with the free state of the Ising model on $\mathbb{Z}^d$, where it is known that the free state is extremal if $\beta \leq \beta_c(\mathbb{Z}^d)$ \cite{ABF87,AF86,ADCS15,Yan52} and is $(\mu^++\mu^-)/2$ where $\mu^+$ and $\mu^-$ are the plus and minus Gibbs states \cite{Aiz80,Hig79,Bod06}.

When $q\geq 3$, rigorous results on the extremality threshold can be found in \cite{Sly11}. In particular, the asymptotic values of thresholds were established for large $d$ when $q\neq 4$. For example, when $q=3$, Theorem \ref{thm} and Theorem 1.1 of \cite{Sly11} together imply that for all sufficiently large $d$, $\mu_{\beta}^{\varnothing}$ is glassy if and only if $(e^{\beta}-1)/(e^{\beta}+2)>1/\sqrt{d}$. 

Motivated by the Edwards-Anderson order parameter~\cite{EA75}, which is the overlap in two independent replicas of the model, we introduce
\begin{equation}\label{eq:Q}
	Q:=\mu_{\beta}^{\varnothing} \otimes \mu_{\beta}^{\varnothing}\left(\{(\eta^{(1)},\eta^{(2)}): \mu^{\eta^{(1)}}=\mu^{\eta^{(2)}}\}\right),
\end{equation}
which measures the overlap in two independent extremal decompositions of $\mu_{\beta}^{\varnothing}$ (see \eqref{eq:muetadef} below for the exact definition of $\mu^{\eta}$). It is clear that $Q=1$ if $\beta$ is small (or more precisely, when $\mu_{\beta}^{\varnothing}$ is extremal), the main results of \cite{GMRS20,CKLN26} imply that $Q=0$ for all large $\beta$. The key novelty of the current paper is to prove a zero-one law for this overlap $Q$. We first derive a necessary and sufficient condition for two extremal Gibbs measures to be equal (see Proposition \ref{Prop:r} below). Then we study the conditional overlaps by conditioning on the spin values at the roots. The tree structure, Proposition \ref{Prop:r}, and Edwards-Sokal coupling \cite{ES88} enable us to establish a system of equations on these conditional overlaps. The zero-one law follows by solving these equations. We have organized our arguments so that they can be applied to broader models, which include the Potts model with external field (see Section~\ref{sec:gen}).

\section{Preliminaries}
In this section, we prove some basic properties on extremal measures which will be useful in the proof of our main result.

We fix an arbitrary vertex $o \in V(\mathbb{T}^d)$ and call it the \textit{root}. For $x,y\in V(\mathbb{T}^d)$, the distance $d(x,y)$ on $\mathbb{T}^d$ is the number of edges in the unique path consisting of nearest-neighbor edges from $x$ to $y$. One may define a partial order on $\mathbb{T}^d$ by $y\succcurlyeq x$ if and only if the unique path from the root $o$ to $y$ passes through $x$. For each $\Lambda \Subset V(\mathbb{T}^d)$, we define the set of boundary vertices by
\[\partial \Lambda:=\{ u\in V(\mathbb{T}^d)\setminus \Lambda: \exists v \in \Lambda \text{ such that } d(u,v)=1\}.\]

Let $\Omega_0:=\{1,\dots,q\}$, $\Omega_{\Lambda}:=\Omega_0^{\Lambda}$ and $\Omega:=\Omega_0^{V(\mathbb{T}^d)}$. For each $\Lambda \subset V(\mathbb{T}^d)$, let $\mathcal{F}_{\Lambda}$ be the $\sigma$-algebra generated by all \textit{local events} in $\Lambda$ (i.e., all events that depend on finitely many spins, all of which are located in $\Lambda$).  We write $\mathcal{F}$ for $\mathcal{F}_{V(\mathbb{T}^d)}$. The \textit{tail $\sigma$-algebra} is defined by $\mathcal{F}_{\infty}:=\cap_{\Lambda \Subset V(\mathbb{T}^d)} \mathcal{F}_{\Lambda^c}$ where the intersection is over all finite subsets.  For each $x\in V(\mathbb{T}^d)$, let $\mathbb{T}^d_x$ denote the subtree rooted at $x$, which is the tree induced by the vertex set $\{x\}\cup\{y\in V(\mathbb{T}^d): y \succcurlyeq x\}$. For $n\in \mathbb{N}\cup \{0\}$, let 
\[W_n:=\{x\in V(\mathbb{T}^d): d(x,o)=n\}\]
be the vertices on the $n$th level, and 
\[V_n:=\cup_{k=0}^n W_k = \{x\in V(\mathbb{T}^d): d(x,o)\leq n\}.\]

If $\mu \in \mathscr{G}_{\beta}$, then the DLR equations \eqref{eq:DLR} imply that
\[\mu(A \cap B)=\int_B \mu_{\Lambda,\beta}^{\eta}(A) \mu(d\eta),\ \forall \Lambda\Subset V(\mathbb{T}^d), \forall A\in \mathcal{F}_{\Lambda}, \forall \text{ local event }B\in \mathcal{F}_{\Lambda^c}.\]
On the other hand, by the definition of the conditional expectation,
\begin{equation*}
	\mu(A \cap B)=\int_B \mu(A | \mathcal{F}_{\Lambda^c})(\eta) \mu(d\eta).
\end{equation*}
Then the almost sure uniqueness of the conditional expectation implies that for each local event $A\in \mathcal{F}_{\Lambda}$,
\begin{equation}\label{eq:DLRalt}
	\mu(A | \mathcal{F}_{\Lambda^c})(\eta)=\mu_{\Lambda,\beta}^{\eta}(A),\ \mu \text{-almost all }\eta.
\end{equation}

The backward martingale convergence theorem gives that for each local event $A$,
\begin{equation}\label{eq:backmart}
	\lim_{n\rightarrow\infty} \mu(A | \mathcal{F}_{V_n^c})(\eta)=\mu(A | \mathcal{F}_{\infty})(\eta), \ \mu \text{-almost all }\eta.
\end{equation}
Since $\{\mu_{V_n,\beta}^{\eta}: n\in \mathbb{N}\}$ is sequentially compact (see, e.g., Theorem 6.24 of \cite{FV18}), for each $\mu \in \mathscr{G}_{\beta}$, we just proved that the following convergence of measures (when evaluated at local events) is well defined for $\mu$-almost all $\eta$,
\begin{equation}\label{eq:muetadef}
	\lim_{n\rightarrow \infty}\mu_{V_n,\beta}^{\eta}=:\mu^{\eta}.
\end{equation}

\begin{proposition}\label{Prop:extrequi}
	Let $\mu \in \mathscr{G}_{\beta}$. Then
	\begin{enumerate}[(i)]
		\item The $\mu^{\eta}$ defined in \eqref{eq:muetadef} is in $\textnormal{ex}\mathcal{G}_{\beta}$ for $\mu$-almost all $\eta$ and
		\begin{equation}\label{eq:extremaldecom1}
			\mu(\cdot)=\int_{\Omega} \mu^{\eta}(\cdot) \mu(d\eta).
		\end{equation}
		\item $\mu\in \textnormal{ex}\mathscr{G}_{\beta}$ if and only if $\lim_{n\rightarrow \infty}\mu_{V_n, \beta}^{\eta}=\mu$ for $\mu$-almost all $\eta$.
	\end{enumerate}
\end{proposition}
Proposition \ref{Prop:extrequi} is known in the literature; see, for example, \cite{Miy74}. We include a proof for completeness.
\begin{proof}
	Equations \eqref{eq:DLRalt}, \eqref{eq:backmart}, \eqref{eq:muetadef} imply that for $\mu$-almost all $\eta$,
	\[\mu^{\eta}(A)=\mu(A | \mathcal{F}_{\infty})(\eta), \ \forall \text{ local event }A.\]
	Since local events $A$ generate the $\sigma$-algebra $\mathcal{F}$, the above equality holds for all events in $\mathcal{F}$. In particular, for $\mu$-almost all $\eta$,
	\[\mu^{\eta}(B)=\mu(B | \mathcal{F}_{\infty})(\eta)=\textbf{1}_{B}(\eta)\in\{0,1\}, \ \forall B\in \mathcal{F}_{\infty}.\]
	So $\mu^{\eta}$ is trivial on $\mathcal{F}_{\infty}$, and thus $\mu^{\eta}\in \text{ex}\mathcal{G}_{\beta}$ by Theorem 6.58 of \cite{FV18} or Theorem 7.7 of~\cite{Geo11}. The integral representation \eqref{eq:extremaldecom1} follows from \eqref{eq:DLR} and \eqref{eq:muetadef}. This completes the proof of the proposition since part (ii) follows directly from part (i).
\end{proof}

For each $\eta\in \Omega$, $x\in V(\mathbb{T}^d)$, $n\geq d(x,o)$, we define the \textit{likelihood ratios}
\begin{equation}
	r_{n,x}^{(i)}(\eta):=\frac{\tilde{\mu}^{\eta}_{\mathbb{T}^d_x \cap V_n,\beta}(\sigma_x=i)}{\tilde{\mu}^{\eta}_{\mathbb{T}^d_x \cap V_n,\beta}(\sigma_x=1)}, \ \forall i\in \Omega_0,
\end{equation}
where $\tilde{\mu}^{\eta}_{\mathbb{T}^d_x \cap V_n,\beta}$ is the same as $\mu^{\eta}_{\mathbb{T}^d_x \cap V_n,\beta}$ defined in \eqref{eq:Pottsdef} except that we require the boundary conditions $\eta$ only act on $\mathbb{T}_x^d \cap W_{n+1}$ (that is, we assume the parent of $x$ (if it exists) is not a boundary point in this measure).

\begin{proposition}\label{Prop:r}
	Let $\mu\in \mathscr{G}_{\beta}$. Then
	\begin{enumerate}[(i)]
		\item $\mu \in \textnormal{ex}\mathscr{G}_{\beta}$ if and only if there exists $(r_x^{(i)})_{ x \in \mathbb{T}^d, i \in \Omega_0}$ such that for $\mu$-almost all $\eta$,
		\[\lim_{n\rightarrow \infty} r_{n,x}^{(i)}(\eta)= r_x^{(i)}, \forall x\in \mathbb{T}^d, i\in \Omega_0.\]
		\item If $\eta$ and $\tilde{\eta}$ are sampled from $\mu$ such that $\mu^{\eta}$ and $\mu^{\tilde{\eta}}$ (defined by \eqref{eq:muetadef}) are well-defined and in $\textnormal{ex}\mathscr{G}_{\beta}$, then $\mu^{\eta}=\mu^{\tilde{\eta}}$ if and only if $\vec{r}_x(\eta)=\vec{r}_x(\tilde{\eta})$ for each $x \in \mathbb{T}^d$ with $d(x,o) \geq N_0$ for some $N_0 \in \mathbb{N}\cup\{0\}$ where $\vec{r}_x(\eta):=(r_x^{(1)},\dots, r_x^{(q)})$ with $r_x^{(i)}$ defined as in part (i) (with the $\mu$ in part (i)  $=\mu^{\eta}$).
	\end{enumerate}
\end{proposition}
Part (i) of Proposition \ref{Prop:r} for the $q=2$ case was Lemma 3 of \cite{Hig77}.
\begin{proof}
	Suppose $\mu \in \textnormal{ex}\mathscr{G}_{\beta}$. Then Proposition \ref{Prop:extrequi} implies that for each $xy\in E(\mathbb{T}^d)$ with $d(y,o)=d(x,o)+1$, we have that for $\mu$-almost all $\eta$,
	\[\lim_{n\rightarrow\infty} \mu_{V_n,\beta}^{\eta}(\sigma_y=i | \sigma_x=1)=\mu(\sigma_y=i | \sigma_x=1), \ \forall i\in \Omega_0.\]
	A direct computation gives
	\[\mu_{V_n,\beta}^{\eta}(\sigma_y=i | \sigma_x=1)=\frac{\exp[\beta \delta_{i,1}] r_{n,y}^{(i)}(\eta)}{\sum_{j=1}^q \exp[\beta \delta_{j,1}] r_{n,y}^{(j)}(\eta)}, \ \forall i \in \Omega_0, n > d(y,o).\]
	The last two displayed equations imply that for $\mu$-almost all $\eta$, the limit $\lim_{n\rightarrow \infty} r_{n,y}^{(i)}(\eta)$ exists for each $y\in V(\mathbb{T}^d)\setminus \{o\}$, $i \in \Omega_0$ and is independent of $\eta$. The existence of the limit $\lim_{n\rightarrow \infty} r_{n,o}^{(i)}(\eta)$ is obvious. 
	
	For each $m<n \in \mathbb{N}$,  each $s\in \Omega_0^{V_m}$, each $\eta \in \Omega$, we have
	\[\mu_{V_n,\beta}^{\eta}(\sigma_u=s_u, \forall u\in V_m)=\frac{\exp[-H_{V_{m-1}}^s(s)]\prod_{x\in W_m} r_{n,x}^{(s_x)}(\eta)}{\sum_{\tau \in \Omega_0^{V_m}}\exp[-H_{V_{m-1}}^{\tau}(\tau)]\prod_{x\in W_m} r_{n,x}^{(\tau_x)}(\eta)},\]
	where 
	\[-H_{V_{m-1}}^{\tau}(\tau):=\beta \sum_{u,v \in V_m: u \sim v }\delta_{\tau_u,\tau_v}.\]
	If we define (note that $r_{n,x}^{(i)}(\eta)$ is uniformly bounded below by $\exp[-(d+1)\beta]$)
	\[h_{n,x}^{(i)}(\eta):=\ln r_{n,x}^{(i)}(\eta), \ \forall x\in V(\mathbb{T}^d), i \in \Omega_0,\]
	then we obtain
	\[\mu_{V_n,\beta}^{\eta}(\sigma_u=s_u, \forall u\in V_m)=\frac{\exp[\beta \sum_{uv\in E(V_m)}\delta_{s_u,s_v}+\sum_{x\in W_m}h_{n,x}^{(s_x)}(\eta)]}{\sum_{\tau \in \Omega_0^{V_m}}\exp[\beta\sum_{uv\in E(V_m)}\delta_{\tau_u,\tau_v}+\sum_{x\in W_m}h_{n,x}^{(\tau_x)}(\eta)]}.\]
	Under the assumption that $\lim_{n\rightarrow \infty}r_{n,x}^{(i)}(\eta)=r_x^{(i)}$ for all $x\in \mathbb{T}^d$ and all $i\in \Omega_0$, we can define that for $\mu$-almost all $\eta$,
	\[h_x^{(i)}:=\ln r_x^{(i)}=\lim_{n\rightarrow \infty} h_{n,x}^{(i)}(\eta), \ \forall x\in \mathbb{T}^d, i \in \Omega_0.\] 
	The last two displayed equations imply that for $\mu$-almost all $\eta$,
	\begin{equation}\label{eq:findis}
		\lim_{n\rightarrow\infty} \mu_{V_n,\beta}^{\eta}(\sigma_u=s_u, \forall u\in V_m) \propto \exp\left[\beta \sum_{uv\in E(V_m)} \delta_{s_u,s_v}+\sum_{x\in W_m} h_x^{(s_x)}\right], \ \forall m\in \mathbb{N}, s\in \Omega_0^{V_m}.
	\end{equation}
	So we conclude that the $\mu^{\eta}$ defined in \eqref{eq:muetadef} is the same measure (say $\nu$) for $\mu$-almost all $\eta$. Combining with \eqref{eq:extremaldecom1}, we get $\nu=\mu$. Part (ii) of Proposition \ref{Prop:extrequi} then implies that $\mu \in \text{ex}\mathscr{G}_{\beta}$. This completes the proof of part (i).
	
	For the proof of part (ii), the forward direction follows trivially from part (i). For the backward direction, note that \eqref{eq:findis} implies that
	\[\mu^{\eta}(\sigma_u=s_u, \forall u\in V_m)\propto \exp\left[\beta \sum_{uv\in E(V_m)} \delta_{s_u,s_v}+\sum_{x\in W_m} h_x^{(s_x)}\right], \ m \geq N_0, s\in \Omega_0^{V_m}.\]
	This means that $(\vec{r}_x(\eta)$  with $x\in \mathbb{T}^d$ and $d(x,o)\geq N_0)$ uniquely determines $\mu^{\eta}$ and thus completes the proof.
\end{proof}

\section{Proof of the main result}
We first prove a zero-one law for the order parameter defined in \eqref{eq:Q}.
\begin{proposition}\label{Prop:01law}
	For each $\beta\geq 0$, we have
	\[\mu_{\beta}^{\varnothing} \otimes \mu_{\beta}^{\varnothing}\left(\{(\eta^{(1)},\eta^{(2)}): \mu^{\eta^{(1)}}=\mu^{\eta^{(2)}}\}\right)\in \{0,1\},\]
	where $\mu_{\beta}^{\varnothing} \otimes \mu_{\beta}^{\varnothing}$ is the product measure.
\end{proposition}

Let $\tilde{\mathbb{T}}^d$ be the infinite $d$-ary tree, i.e., an infinite tree with a root $o$ having degree $d$ and all other vertices having degree $d+1$. Note that Propositions \ref{Prop:extrequi} and \ref{Prop:r} hold for Gibbs measures defined on $\tilde{\mathbb{T}}^d$ with the same proofs. We may label the root of $\tilde{\mathbb{T}}^d$ as $\emptyset$; label its $d$ children using $\{1,2,\dots,d\}$. Every vertex at depth $n$ is uniquely represented as a word of length $n$ (like $u=a_1a_2\dots a_n$ with $a_i\in \{1,2,\dots, d\}$). For $x\in V(\tilde{\mathbb{T}}^d)$, let $\tilde{\mathbb{T}}_{x}^d$ be the subtree rooted at $x$. Define $T: \tilde{\mathbb{T}}^d \rightarrow \tilde{\mathbb{T}}_{x_1}^d$ with $T(u)=1\cdot u$ by prepending the letter $1$ to the vertex $u$. Then $T$ is clearly a graph isomorphism between $\tilde{\mathbb{T}}^d$ and $\tilde{\mathbb{T}}_{x_1}^d$ (i.e., a bijection between the vertex sets that preserves the neighboring relationships); see Figure~\ref{fig:shift}. Similarly, for $x\in W_1$, let $S: \tilde{\mathbb{T}}^d \rightarrow \mathbb{T}_x^d$ be a graph isomorphism defined as in Figure \ref{fig:shift}. Let $\mu_{\tilde{\mathbb{T}}^d}^{\varnothing}$ be the free state defined on $\tilde{\mathbb{T}}^d$ at the fixed inverse temperature $\beta\geq0$. 

\begin{figure}
	\begin{center}
		\includegraphics[width=\textwidth]{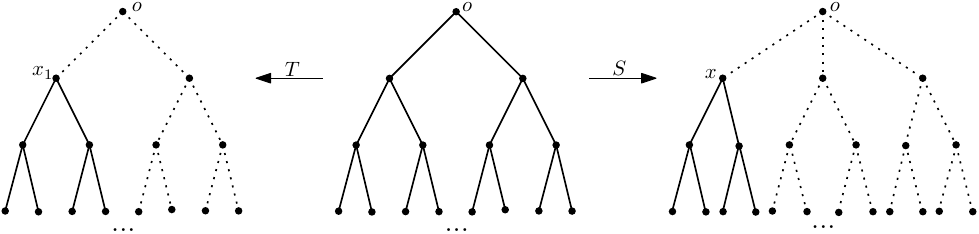}
		\caption{The graph isomorphism $T: \tilde{\mathbb{T}}^2 \rightarrow \tilde{\mathbb{T}}_{x_1}^2$ (where $\tilde{\mathbb{T}}_{x_1}^2$ is a subtree of $\tilde{\mathbb{T}}^2$ and is represented by solid edges on the left) and $S: \tilde{\mathbb{T}}^2 \rightarrow \mathbb{T}_x^2$ (where $\mathbb{T}_x^2$ is a subtree of $\mathbb{T}^2$ and is represented by solid edges on the right).}\label{fig:shift}
	\end{center}
\end{figure}

\begin{lemma}\label{lem:3prop}
		\begin{enumerate}[(1)]
		\item Under the conditional measure $\mu_{\tilde{\mathbb{T}}^d}^i:=\mu_{\tilde{\mathbb{T}}^d}^{\varnothing}(\cdot \mid \eta_o=i)$, $\left. \eta \right |_{\tilde{\mathbb{T}}_x^d}$  for all $x\in V(\tilde{\mathbb{T}}^d)$ with $d(x,o)=1$ are independent and identically distributed.
		\item For $x_1\in V(\tilde{\mathbb{T}}^d)$ with the label $1$, $\mu_{\tilde{\mathbb{T}}^d}^i(\cdot \mid \eta_{x_1}=j)$ restricted to $\tilde{\mathbb{T}}_{x_1}^d$ and $\mu_{\tilde{\mathbb{T}}^d}^j$ have the same distribution up to the shift $T$:
		\[\mu_{\tilde{\mathbb{T}}^d}^i(T(E) \mid \eta_{x_1}=j) = \mu_{\tilde{\mathbb{T}}^d}^j(E), \ \forall E \in \mathcal{F}_{V(\tilde{\mathbb{T}}^d)},\]
		where 
		\[T(E):=\{\sigma: \sigma_{Tu}=\eta_u, \forall u \in V(\tilde{\mathbb{T}}^d), \eta \in E\}.\]
		\item For each $x\in W_1$ and $s\in \Omega_0^{V_1}$, 
		\begin{equation}\label{eq:shift}
			\mu_{\beta}^{\varnothing}(S(E) \mid \eta_{V_1}=s) = \mu_{\tilde{\mathbb{T}}^d}^{\varnothing}(E \mid \eta_0=s_x),\ \forall E \in \mathcal{F}_{V(\tilde{\mathbb{T}}^d)}.
		\end{equation}
	\end{enumerate}
\end{lemma}
All three properties hold for more general Gibbs measures that we will discuss in Section~\ref{sec:gen}. Here we give an alternative and straightforward proof using the Edwards-Sokal coupling~\cite{ES88}.
\begin{proof}
	Usually, the Edwards-Sokal coupling between the random-cluster model and Potts model is defined on finite graphs, but it can be generalized to infinite graphs (see, e.g., Theorem 4.91 of \cite{Gri06}). Since $\tilde{\mathbb{T}}^d$ has no loops, the random-cluster model on $\tilde{\mathbb{T}}^d$ with edge open probability $1-e^{-\beta}$ is the same as the Bernoulli percolation on $\tilde{\mathbb{T}}^d$ with $\hat{p}=(1-e^{-\beta})/(1-e^{-\beta}+qe^{-\beta})$. To obtain a configuration $\eta$ sampled from $\mu^{\varnothing}_{\tilde{\mathbb{T}}^d}$, we first sample $\omega$ according to the Bernoulli $\hat{p}$-percolation on $E(\tilde{\mathbb{T}}^d)$, then assign a spin uniformly in $\Omega_0$ to each open cluster, and the assignments among different open clusters are independent. Since both bond percolation and spin assignments (among open clusters) are independent, this immediately implies property (1) above.
	
	For property (2), we first note that
	\[\mu_{\tilde{\mathbb{T}}^d}^i(\cdot \mid \eta_{x_1}=j) = \mu_{\tilde{\mathbb{T}}^d}^{\varnothing}(\cdot \mid \eta_{x_1}=j, \eta_0=i). \]
	Let $\eta \in \Omega_0^{\tilde{\mathbb{T}}_{x_1}^d}$ be the restriction of a configuration sampled from $\mu_{\tilde{\mathbb{T}}^d}^{\varnothing}(\cdot \mid \eta_{x_1}=j, \eta_0=i)$ to $\tilde{\mathbb{T}}^d_{x_1}$. We may obtain $\eta$ by first sampling Bernoulli $\hat{p}$-percolation on $E(\tilde{\mathbb{T}}^d_{x_1})$ and then assigning $j$ to the open cluster of $x_1$ and a spin uniformly in $\Omega_0$ to the remaining open clusters.  It is now clear that $T^{-1}\eta$ is distributed as a configuration from $\mu_{\tilde{\mathbb{T}}^d}^j=\mu_{\tilde{\mathbb{T}}^d}^{\varnothing}(\cdot \mid \eta_o=j)$. The proof of property (3) is similar. 
\end{proof}

Define
\begin{equation}\label{eq:c_ijdef}
	\begin{split}
		&c_{ij}:=\mu_{\tilde{\mathbb{T}}^d}^{\varnothing}\otimes \mu_{\tilde{\mathbb{T}}^d}^{\varnothing} \left(\{(\eta^{(1)},\eta^{(2)}): \mu^{\eta^{(1)}}=\mu^{\eta^{(2)}}\} \middle| \eta^{(1)}_o=i, \eta^{(2)}_o=j\right), \ \forall i, j\in \Omega_0,\\
		& p_{ij}:=\mu_{\tilde{\mathbb{T}}^d}^{\varnothing}(\eta_{x_1}=j \mid \eta_{o}=i), \ \forall x_1 \sim o, i,j \in \Omega_0.
	\end{split}
\end{equation}

We next derive a system of equations for $c_{ij}$.
\begin{lemma}\label{lem:eqns}
	For $c_{ij}$ and $p_{ij}$ defined in \eqref{eq:c_ijdef}, we have
	\begin{equation}\label{eq:sysine}
		c_{ij} = \left[\sum_{k=1}^q\sum_{l=1}^q p_{ik} p_{jl} c_{kl}\right]^d, \ \forall i, j \in \Omega_0.
	\end{equation}
\end{lemma}

\begin{proof}
	To simplify the notation, we write $\mu^{\varnothing}$ for $\mu_{\tilde{\mathbb{T}}^d}^{\varnothing}$ in the proof. Part (i) of Proposition~\ref{Prop:extrequi} and Part (ii) of Proposition \ref{Prop:r} imply that
	\[c_{ij}= \mu^{\varnothing}\otimes \mu^{\varnothing}(\vec{r}_x(\eta^{(1)})=\vec{r}_{x}(\eta^{(2)}), \forall x \in \tilde{\mathbb{T}}^d \setminus\{o\} \mid \eta^{(1)}_o=i, \eta^{(2)}_o=j).\]
	Then property (1) from Lemma \ref{lem:3prop} gives
	\[c_{ij}=\left[\mu^{\varnothing}\otimes \mu^{\varnothing}(\vec{r}_x(\eta^{(1)})=\vec{r}_{x}(\eta^{(2)}), \forall x \in \tilde{\mathbb{T}}_{x_1}^d \mid \eta^{(1)}_o=i, \eta^{(2)}_o=j)\right]^d.\]
	The law of total probability yields
	\begin{align*}
		c_{ij}&= \left[\sum_{k=1}^q \sum_{l=1}^q p_{ik} p_{jl} \mu^i \otimes \mu^j (\vec{r}_x(\eta^{(1)})=\vec{r}_{x}(\eta^{(2)}), \forall x \in \tilde{\mathbb{T}}_{x_1}^d \mid \eta^{(1)}_{x_1}=k, \eta^{(2)}_{x_1}=l)\right]^d\\
		&= \left[\sum_{k=1}^q \sum_{l=1}^q p_{ik} p_{jl} \mu^k\otimes \mu^l (\vec{r}_x(\eta^{(1)})=\vec{r}_{x}(\eta^{(2)}), \forall x \in \tilde{\mathbb{T}}^d )\right]^d,
	\end{align*}
	where we have used property (2) from Lemma \ref{lem:3prop} in the last equality. Finally, another application of Propositions~\ref{Prop:extrequi} and \ref{Prop:r} gives
	\[c_{ij}=\left[\sum_{k=1}^q \sum_{l=1}^q p_{ik} p_{jl} c_{kl}\right]^d.\]
	This is exactly the desired result.
\end{proof}

\begin{proof}[Proof of Proposition \ref{Prop:01law}]
	Note that
	\[\sum_{k=1}^q p_{ik}=1, \ \forall i\in \Omega_0 \text{ and }\ p_{i,k}\in (0,1), \ \forall i,k \in \Omega_0.\]
	Let 
	\[M:=\max_{i,j\in \Omega_0} c_{ij}.\]
	Suppose $M=c_{i_0j_0}$ for some $i_0,j_0\in \Omega_0$. Then equation \eqref{eq:sysine} gives that
	\[M=c_{i_0j_0}\leq \left[\sum_{k=1}^q \sum_{l=1}^q p_{i_0k}p_{j_0l} M\right]^d=M^d,\]
	which implies that $M\in \{0,1\}$ and the inequality must be an equality. Since $p_{i_0k}p_{j_0l}>0$ for all $k,l\in\Omega_0$, we obtain
	\[c_{ij}=M\in \{0,1\}, \ \forall i,j\in \Omega_0.\]
	
	Part (i) of Proposition \ref{Prop:extrequi} and Part (ii) of Proposition \ref{Prop:r} imply that
	\begin{align*}
		&\mu_{\beta}^{\varnothing} \otimes \mu_{\beta}^{\varnothing}( \mu^{\eta^{(1)}}=\mu^{\eta^{(2)}})\\
		= &\sum_{s\in \Omega_0^{V_1}}\sum_{t \in \Omega_0^{V_1}}\mu_{\beta}^{\varnothing} \otimes \mu_{\beta}^{\varnothing}(\vec{r}_x(\eta^{(1)})=\vec{r}_x(\eta^{(2)}), \forall x\in \mathbb{T}^d \setminus \{o\} \mid \eta^{(1)}_{V_1}=s, \eta^{(2)}_{V_1}=t) \\
		&\qquad \times \mu_{\beta}^{\varnothing} \otimes \mu_{\beta}^{\varnothing}(\eta^{(1)}_{V_1}=s, \eta^{(2)}_{V_1}=t)\\
		=&\sum_{s\in \Omega_0^{V_1}}\sum_{t \in \Omega_0^{V_1}} \prod_{u \sim o} c_{s_ut_u}\mu_{\beta}^{\varnothing} \otimes \mu_{\beta}^{\varnothing}(\eta^{(1)}_{V_1}=s, \eta^{(2)}_{V_1}=t)=M^{d+1} \in \{0,1\},
	\end{align*}
	where we have used property (3) from Lemma \ref{lem:3prop} in the last equality.
\end{proof}

We are ready to prove Theorem \ref{thm}.
\begin{proof}[Proof of Theorem \ref{thm}]
	We assume that $\mu_{\beta}^{\varnothing}$ is not extremal. Then \eqref{eq:extremaldecom} and \eqref{eq:extremaldecom1} imply that
	\[\mu_{\beta}^{\varnothing}(\cdot)=\int_{\Omega} \mu^{\eta}(\cdot) \mu_{\beta}^{\varnothing} (d\eta)=\int_{\text{ex}\mathscr{G}_{\beta}} \nu(\cdot) \lambda_{\mu^{\varnothing}_{\beta}}(d\nu).\]
	We further assume that $\lambda_{\mu^{\varnothing}_{\beta}}$ has an atom at $\nu_0 \in \text{ex}\mathscr{G}_{\beta}$: 
	\[\epsilon:=\lambda_{\mu^{\varnothing}_{\beta}}(\{\nu_0\})=\mu^{\varnothing}_{\beta}(\{\eta \in \Omega: \mu^{\eta}=\nu_0\})\in(0,1).\]
	Then we have
	\[\mu_{\beta}^{\varnothing}(\cdot)=\epsilon \nu_0+\int_{\text{ex}\mathscr{G}_{\beta}\setminus\{\nu_0\}} \nu(\cdot) \lambda_{\mu^{\varnothing}_{\beta}}(d\nu).\]
	Therefore,
	\[\mu_{\beta}^{\varnothing} \otimes \mu_{\beta}^{\varnothing}\left(\{(\eta^{(1)},\eta^{(2)}): \mu^{\eta^{(1)}}=\mu^{\eta^{(2)}}\}\right) \geq \mu_{\beta}^{\varnothing} \otimes \mu_{\beta}^{\varnothing}(\mu^{\eta^{(1)}}=\mu^{\eta^{(2)}}=\nu_0)=\epsilon^2,\]
	and
	\begin{align*}
		&\quad\mu_{\beta}^{\varnothing} \otimes \mu_{\beta}^{\varnothing}\left(\{(\eta^{(1)},\eta^{(2)}): \mu^{\eta^{(1)}}=\mu^{\eta^{(2)}}\}\right) \leq \mu_{\beta}^{\varnothing} \otimes \mu_{\beta}^{\varnothing}(\mu^{\eta^{(1)}}=\mu^{\eta^{(2)}}=\nu_0)\\
		&\qquad\qquad+\mu_{\beta}^{\varnothing} \otimes \mu_{\beta}^{\varnothing}(\mu^{\eta^{(1)}}\in \text{ex}\mathscr{G}_{\beta}\setminus \{\nu_0\}, \mu^{\eta^{(2)}}\in \text{ex}\mathscr{G}_{\beta}\setminus \{\nu_0\})\\
		&\qquad=\epsilon^2+(1-\epsilon)^2<1,
	\end{align*}
	which contradicts Proposition \ref{Prop:01law}. If the support of $\lambda_{\mu^{\varnothing}_{\beta}}$ were countable, then of course $\lambda_{\mu^{\varnothing}_{\beta}}$ would have an atom. This completes the proof of the theorem.
\end{proof}

\section{Generalizations to the Potts model with external field}\label{sec:gen}
The $q$-state Potts model at inverse temperature $\beta$ on $\Lambda \Subset V(\mathbb{T}^d)$ with uniform external field $\vec{H}:=(H_1,\dots,H_q)\in \mathbb{R}^q$ and boundary conditions $\eta \in \Omega_0^{\Lambda^c}$ is the probability measure
\begin{equation}\label{eq:Pottsdef1}
	\mu^{\eta}_{\Lambda,\beta,\vec{H}}(\sigma) \propto \exp\left[\beta\sum_{u,v\in\Lambda: u\sim v}\delta_{\sigma_u,\sigma_v}+\beta\sum_{u\in\Lambda, v\in \Lambda^c: u\sim v}\delta_{\sigma_u,\eta_v}+\sum_{u\in \Lambda} H_{\sigma_u}\right],\quad \forall \sigma \in \Omega_0^{\Lambda}.
\end{equation}
A measure $\mu$ on $\Omega$ is called a Gibbs measure if it satisfies the DLR equations \eqref{eq:DLR}. Let $\mathscr{G}_{\beta,\vec{H}}$ be the set of all Gibbs measures. The set of all extremal Gibbs measures, $\text{ex}\mathscr{G}_{\beta,\vec{H}}$, is defined in the same way as $\text{ex}\mathscr{G}_{\beta}$. A similar extremal decomposition formula as \eqref{eq:extremaldecom} also holds for $\mu \in \mathscr{G}_{\beta,\vec{H}}$. We focus on the set of \textit{completely homogeneous splitting Gibbs measures (CHSGMs)} $\mu \in \mathscr{G}_{\beta,\vec{H}}$ defined in \cite{BR19}. First, there is $\vec{h}\in \mathbb{R}^q$ such that
\begin{equation}\label{eq:CHSGMdef}
	\mu(\sigma_{V_n}=s) \propto \exp\left[\beta\sum_{uv\in E(V_n)}\delta_{s_u,s_v}+\sum_{u\in V_n}H_{s_u}+\sum_{u\in W_n}h_{s_u}\right], \ \forall n\in \mathbb{N}, s\in \Omega_0^{V_n},
\end{equation}
then by Theorem 2.1 and Corollary 2.7 of \cite{BR19}, $\mu$ is a CHSGM if and only if $\vec{H}$ and $\vec{h}$ satisfy
\begin{equation}\label{eq:CHSGMiff}
	h_i-h_q=d \ln \frac{\exp[\beta+H_i+h_i]+\sum_{j \neq i}\exp[H_j+h_j]}{\exp[\beta+H_q+h_q]+\sum_{j\neq q}\exp[H_j+h_j]}, \ \forall i\in \Omega_0.
\end{equation}

Proposition \ref{Prop:extrequi} is valid for measures in $\mathscr{G}_{\beta,\vec{H}}$ with the same proof; Proposition \ref{Prop:r} is also valid  for measures in $\mathscr{G}_{\beta,\vec{H}}$ but in the proof one needs to include the external field. Let $\mu \in \mathscr{G}_{\beta,\vec{H}}$ be a CHSGM defined by \eqref{eq:CHSGMdef} and \eqref{eq:CHSGMiff}. Let
\[\tilde{V}_n:=\{x\in V(\tilde{\mathbb{T}}^d): d(x,o)\leq n\}, \tilde{W}_n:=\{x\in V(\tilde{\mathbb{T}}^d): d(x,o)= n\}.\]
Let $\mu_{\tilde{\mathbb{T}}^d}$ be a measure defined on $\Omega_0^{V(\tilde{\mathbb{T}}^d)}$ such that
\begin{equation}\label{eq:mutilde}
	\mu_{\tilde{\mathbb{T}}^d}(\sigma_{\tilde{V}_n}=s) \propto \exp\left[\beta\sum_{uv\in E(\tilde{V}_n)}\delta_{s_u,s_v}+\sum_{u\in \tilde{V}_n}H_{s_u}+\sum_{u\in \tilde{W}_n}h_{s_u}\right], \ \forall n\in \mathbb{N}\cup\{0\}, s\in \Omega_0^{\tilde{V}_n}.
\end{equation}
Then it is easy to check that $\mu_{\tilde{\mathbb{T}}^d}$ is a Gibbs measure (see, e.g., p. 7 of \cite{BR19} for a similar proof). The following lemma is analogous to Lemma \ref{lem:3prop} but holds for general CHSGMs $\mu$ and their corresponding $\mu_{\tilde{\mathbb{T}}^d}$.

\begin{lemma}\label{lem:3propfield}
	\begin{enumerate}[(1)]
		\item Under the conditional measure $\mu_{\tilde{\mathbb{T}}^d}^i:=\mu_{\tilde{\mathbb{T}}^d}(\cdot \mid \eta_o=i)$, $\left. \eta \right |_{\tilde{\mathbb{T}}_x^d}$  for all $x\in \tilde{W}_1$ are independent and identically distributed.
		\item For each $x_1\in \tilde{W}_1$, $\mu_{\tilde{\mathbb{T}}^d}^i(\cdot \mid \eta_{x_1}=j)$ restricted to $\tilde{\mathbb{T}}_{x_1}^d$ and $\mu_{\tilde{\mathbb{T}}^d}^j$ have the same distribution up to the shift $T$:
		\[\mu_{\tilde{\mathbb{T}}^d}^i(T(E) \mid \eta_{x_1}=j) = \mu_{\tilde{\mathbb{T}}^d}^j(E), \ \forall E \in \mathcal{F}_{V(\tilde{\mathbb{T}}^d)}.\]
		\item For each $x\in W_1$ and $s\in \Omega_0^{V_1}$,
		\begin{equation}
			\mu(S(E) \mid \eta_{V_1}=s) = \mu_{\tilde{\mathbb{T}}^d}(E \mid \eta_o=s_x),\ \forall E \in \mathcal{F}_{V(\tilde{\mathbb{T}}^d)}.
		\end{equation}
	\end{enumerate}
\end{lemma}
\begin{proof}
	Property (1) is the splitting property, which can be derived directly from~\eqref{eq:mutilde}. Property (2) follows since by \eqref{eq:mutilde},
	\begin{align*}
		&\mu_{\tilde{\mathbb{T}}^d}^i(\eta_{\tilde{V}_n \cap V(\tilde{\mathbb{T}}_{x_1}^d)}=Ts \mid \eta_{x_1}=j) = \mu_{\tilde{\mathbb{T}}^d}(\eta_{\tilde{V}_n \cap V(\tilde{\mathbb{T}}_{x_1}^d)}=Ts \mid \eta_{x_1}=j, \eta_{o}=i)\\
		&\qquad =\mu_{\tilde{\mathbb{T}}^d}(\eta_{\tilde{V}_{n-1}}=s \mid \eta_o=j), \ \forall n\in \mathbb{N}, s\in \Omega_0^{\tilde{V}_{n-1}} \text{ with } s_o=j,
	\end{align*}
	where $(Ts)_{Tu}:=s_{u}$ for each $u\in \tilde{V}_{n-1}$. Property (3) follows because by \eqref{eq:CHSGMdef} and \eqref{eq:mutilde},
	\begin{align*}
		&\mu(\eta_{V_n \cap V(\mathbb{T}^d_x)}=St \mid \eta_{V_1}=s) \\
		&\qquad \propto \exp\left[\beta \sum_{uv\in E(V_n \cap \mathbb{T}^d_x)}\delta_{(St)_u,(St)_v}+\sum_{u\in V_n \cap V(\mathbb{T}^d_x)} H_{(St)_u}+\sum_{u\in W_n \cap V(\mathbb{T}^d_x)} h_{(St)_u}\right]\\
		&\qquad \propto \mu_{\tilde{\mathbb{T}}^d}(\eta_{\tilde{V}_{n-1}}=t \mid \eta_0=s_x), \ \forall n\in \mathbb{N}, t\in \Omega_0^{\tilde{V}_{n-1}} \text{ with } t_{o}=s_x.
	\end{align*}
	This completes the proof of the lemma.
\end{proof}

So Proposition \ref{Prop:01law} holds with $\mu_{\beta}^{\varnothing}$ replaced by any CHSGM $\mu$, and thus the same proof as Theorem \ref{thm} gives
\begin{theorem}\label{thm:gen}
	Let $\mu$ be a completely homogeneous splitting Gibbs measure for the $q$-state Potts model on $\mathbb{T}^d$ at inverse temperature $\beta\geq 0$ and uniform external field $\vec{H}\in \mathbb{R}^q$ defined by \eqref{eq:CHSGMdef} and \eqref{eq:CHSGMiff}. Then $\mu$ satisfies the following dichotomy:
	\begin{enumerate}[(i)]
		\item either $\mu$ is extremal, i.e., $\mu \in \textnormal{ex}\mathscr{G}_{\beta,\vec{H}}$,
		\item or the extremal decomposition measure $\lambda_{\mu}$ has no atoms, i.e., $\lambda_{\mu}(\{\nu\})=0$ for each $\nu\in \textnormal{ex}\mathscr{G}_{\beta,\vec{H}}$. In particular, the support of $\lambda_{\mu}$ contains uncountably many extremal measures.
	\end{enumerate}
\end{theorem}

Of course the free state $\mu_{\beta}^{\varnothing}$ in Theorem \ref{thm} is a CHSGM with $\vec{H}=\vec{h}=\vec{0}$ by \eqref{eq:Kol}, so Theorem \ref{thm} is a particular case of Theorem \ref{thm:gen}. For the classical Ising model (corresponding to $q=2$) with uniform external field $H$, that is,
\[\mu_{\Lambda,\beta,H}^{\eta}(\sigma)\propto \exp\left[\beta\sum_{uv \in E(\Lambda)}\sigma_u\sigma_v+\beta\sum_{uv: u\in \Lambda, v\in \Lambda^c}\sigma_u\sigma_v+H\sum_{u\in \Lambda}\sigma_u\right], \ \forall \sigma\in \{-1,+1\}^{\Lambda}.\]
It is known since \cite{Pre74} (see also \cite{Hig77} and Theorem 2.8 of \cite{BR19}) that, for each $\beta >\arctanh(1/d)$, there exists $H_c(\beta)>0$ such that 
\begin{equation*}
	\text{the total number of CHSGMs is }\begin{cases}
		2, & \beta>\arctanh(1/d) \text{ and }H=\pm H_c(\beta),\\
		3, & \beta>\arctanh(1/d) \text{ and }|H|<H_c(\beta),\\
		1, & \text{otherwise}.
	\end{cases}
\end{equation*}
When there are $3$ CHSGMs, it is still an open question to determine the exact $(\beta,H)$ region such that the middle state is nonextremal (and thus glassy by Theorem \ref{thm:gen}). We refer to \cite{KRK14,KR17,BR19} for more results on the CHSGMs for general $q$ where the CHSGMs are also known as the translation-invariant splitting Gibbs measures (TISGMs)  in \cite{KRK14,KR17}.

\section*{Acknowledgments}
This research was partially supported by National Natural Science Foundation of China (No. 12271284 and No. 12226001). The authors thank Chuck Newman and Dan Stein for valuable comments. They also thank Utkir Rozikov for providing useful references.

\bibliographystyle{abbrv}
\bibliography{reference}

\end{document}